\documentclass[12pt]{amsart}

\usepackage[english]{babel}
\usepackage{vmargin,graphicx,amsmath,amssymb,color,subfigure,enumerate,csquotes, dsfont, mathrsfs}
\numberwithin{equation}{section}

\usepackage{color}
\definecolor{gr}{rgb}   {0.,   0.69,   0.23 }
\definecolor{bl}{rgb}   {0.,   0.5,   1. }
\definecolor{mg}{rgb}   {0.85,  0.,    0.85}
\definecolor{or}{rgb}   {0.9,  0.5,   0.}

\definecolor{webred}{rgb}{0.75,0,0}
\definecolor{webgreen}{rgb}{0,0.75,0}
\usepackage[citecolor=webgreen,colorlinks=true,linkcolor=webred]{hyperref}

\newtheorem{theorem}{Theorem}[section]
\newtheorem{lemma}[theorem]{Lemma}
\newtheorem{notation}[theorem]{Notation}
\newtheorem{remark}[theorem]{Remark}

\newtheorem{proposition}[theorem]{Proposition}

\newcommand{\Bl}{\color{blue}}

\newcommand{\dist}{\mathsf{dist}}

\newcommand{\x}{\mathbf{x}}

\newcommand{\n}{\mathbf{n}}
\newcommand{\Dom}{\mathsf{Dom}}
\newcommand{\supp}{\mathsf{supp}\,}

\newcommand{\h}{\hbar}

\newcommand{\R}{\mathbb{R}}
\newcommand{\C}{\mathbb{C}}

\newcommand{\eps}{\varepsilon}
\newcommand{\dx}{\,\mathrm{d}}

\begin{document}

\title[]{Weyl formulae for the Robin Laplacian \\ in the semiclassical limit}

\author{A. Kachmar}
\address[A. Kachmar]{Lebanese University, Department of Mathematics, Hadath, Lebanon.}
\email{ayman.kashmar@gmail.com}

\author{P. Keraval}
\address[P. Keraval]{IRMAR, Universit\'e de Rennes 1, Campus de Beaulieu, F-35042 Rennes cedex, France}
\email{pierig.keraval@gmail.com}

\author{N. Raymond}
\address[N. Raymond]{IRMAR, Universit\'e de Rennes 1, Campus de Beaulieu, F-35042 Rennes cedex, France}
\email{nicolas.raymond@univ-rennes1.fr}

\date{\today}

\maketitle
\begin{abstract}
This paper is devoted to establish semiclassical Weyl formulae for the Robin Laplacian on smooth domains in any dimension. Theirs proofs are reminiscent of the Born-Oppenheimer method.
\end{abstract}
\section{Introduction}

\subsection{Context and motivations}
For $d\geq 2$, let us consider an \emph{open bounded connected} subset of $\R^d$ denoted by $\Omega$ with a $\mathcal{C}^3$ connected boundary $\Gamma=\partial\Omega$ and for which the standard tubular coordinates are well defined (see Section \ref{sec.bnd}). On this domain, we consider the Robin Laplacian $\mathcal{L}_{h}$ defined as the self-adjoint operator associated with the closed quadratic form defined on $H^1(\Omega)$ by the formula
\[\forall u\in H^1(\Omega)\,,\qquad\mathcal{Q}_{h}(u)=\int_{\Omega}|h\nabla u|^2 \dx\x-h^{\frac{3}{2}}\int_{\Gamma}|u|^2\dx\Gamma\,,\]
where $\dx\Gamma$ is the surface measure of the boundary and $h>0$ is the semiclassical parameter. The domain of the operator $\mathcal{L}_{h}$ is given by
\[\Dom(\mathcal{L}_{h})=\{u\in H^2(\Omega) : \n\cdot h^{\frac{1}{2}}\nabla u=-u\mbox{ on }\Gamma\}\,,\]
where $\n$ is the inward pointing normal to the boundary.

The aim of this paper is the quantify the number of non positive eigenvalues created by the Robin condition in the semiclassical limit $h\to 0$. The estimate of the non positive spectrum of the Robin Laplacian in the semiclassical limit (or equivalently in the strong coupling limit) has given rise to many contributions (in various geometric contexts) in the last years (see \cite{LP, EMP, HK-tams, PP-eh, HKR15}). Negative eigenvalues of the operator $\mathcal L_h$ have eigenfunctions localized near the boundary of the domain thereby serving as  edge states. One of the most characteristic results is established in \cite{PP-eh} and states that the $n$-th eigenvalue of $\mathcal{L}_{h}$ is approximated, modulo $\mathcal{O}(h^2)$, by the $n$-th eigenvalue of the effective Hamiltonian acting on the boundary
\begin{equation}\label{eq.Heff}
-h+h^2\mathcal{L}^\Gamma-h^{\frac{3}{2}}\kappa\,,
\end{equation}
where $\mathcal{L}^\Gamma$ is the Laplace-Beltrami operator on $\Gamma$ and where $\kappa$ is the mean curvature. The approximation of the eigenfunctions of $\mathcal L_h$ via those of the effective Hamiltonian is obtained in \cite{HK-tams} for the two dimensional situation.

Let us emphasize here that the effective Hamiltonian in \eqref{eq.Heff} concerns \emph{individual} eigenvalues and thus that it is not direct to deduce (more than formally) an asymptotic estimate of the counting function of $\mathcal{L}_{h}$. In two dimensions, the problem of deriving a strengthened effective Hamiltonian was also tackled to investigate semiclassical tunneling in presence of symmetries in \cite{HKR15}. Moreover, in \cite[Section 7]{HKR15}, as a byproduct of the strategy developed there (which was initially inspired by \cite{MT05, DR14} or \cite{Ray}), Weyl formulae are established in two dimensions. The present paper is an extension of these results to any dimension and it proves, in an appropriate energy window, a \emph{uniform} approximation of $\mathsf{sp}\left(\mathcal{L}_{h}\right)$ by the spectrum of a slight perturbation of  \eqref{eq.Heff}.

\subsection{Results}
For $\lambda\in\R$, we denote by
\[\mathsf{N}\left(\mathcal{L}_{h},\lambda\right)={\rm Tr}\Big(\mathbf 1_{(-\infty,\lambda]}(\mathcal L_h)\Big)\,,\]
the number of eigenvalues $\mu_n(h)$ of $\mathcal L_h$ below the energy level $\lambda$. Let us now state our main two theorems that relate the counting functions of $\mathcal{L}_{h}$ and $\mathcal{L}^\Gamma$ in the semiclassical limit.
\begin{theorem}\label{theo.Weyl1}
We have the following Weyl estimate for the low lying eigenvalues:
\[\forall E \in \mathbb R\,,\,\quad \mathsf{N}\left(\mathcal{L}_{h},-h+Eh^{\frac{3}{2}}\right)\underset{h\to 0}{\sim}\mathsf{N}\left(h^{\frac{1}{2}}\mathcal{L}^\Gamma-\kappa, E\right)\,.\]

\end{theorem}

\begin{theorem}\label{theo.Weyl2}
We have the following Weyl estimate for the non positive eigenvalues:
\[\mathsf{N}\left(\mathcal{L}_{h},0\right)\underset{h\to 0}{\sim}\mathsf{N}\left(h\mathcal{L}^\Gamma,1\right)\,.\]
\end{theorem}
\begin{remark}
\rm Note that we have the classical Weyl estimates (see for instance \cite[Theorem 14.11]{Z13}):
\begin{equation}\label{eq.Weyl-stand'}
\mathsf{N}\left(h^{\frac{1}{2}}\mathcal{L}^\Gamma-\kappa, E\right)\underset{h\to 0}{\sim}\frac{1}{\left(2\pi h^{\frac{1}{4}}\right)^{d-1}}\mathsf{Vol}_{T^*\Gamma}\{(s,\sigma)\, :\, |\sigma|^2_{g}-\kappa(s)\leq E\} \,,
\end{equation}
\begin{equation}\label{eq.Weyl-stand}
\mathsf{N}\left(h\mathcal{L}^\Gamma,1\right)\underset{h\to 0}{\sim}\frac{1}{\left(2\pi h^{\frac{1}{2}}\right)^{d-1}}\mathsf{Vol}_{T^*\Gamma}\{(s,\sigma)\, :\, |\sigma|^2_{g}\leq 1\} \,.
\end{equation}
Note that they remain true if $E$ and $1$ are replaced by $E+o(1)$ and $1+o(1)$ respectively.
\end{remark}

\begin{remark}
\rm Let us notice here that these results are proved in the case of a $\mathcal{C}^3$ bounded and connected boundary. The connectedness is actually not necessary but avoids to consider each connected component separately. For Theorem~\ref{theo.Weyl1}, the boundedness of $\Gamma$ is not necessary either (bounds on the curvature are enough), but allows a lighter presentation. We refer to \cite{PP-eh} where such geometric assumptions are accurately described. 
\end{remark}

\begin{remark}
\rm The proof we give to Theorem~\ref{theo.Weyl2} uses the classical Weyl law in the interior of the domain $\Omega$. This law requires that the domain $\Omega$ is bounded.
\end{remark}

\subsection{Strategy of the proofs}
In Section \ref{sec.red-bnd}, we show that the interior of $\Omega$ does not contribute to the creation of non positive spectrum (the Laplacian is non negative inside $\Omega$). We quantify this thanks to classical Agmon estimates and reduce the investigation to a Robin Laplacian on a thin neighborhood of the boundary (see Proposition \ref{prop:red-bnd}). In Section \ref{sec.eff}, by using an idea from the Born-Oppenheimer context, we derive \emph{uniform} effective Hamiltonians (see Theorem \ref{theo.eff}) whose eigenvalues simultaneously describe the eigenvalues of $\mathcal{L}_{h}$ less than $-\eps_{0}h$ (for $\eps_{0}>0$ as small as we want). In particular, we show that the effectiveness of the reduction to a boundary operator is determined by the estimate of the \emph{Born-Oppenheimer correction}. This correction is an explicit quantity related to dimension one. It appears in physics and, for instance, in the contributions \cite{MS02, PST03, T03, PST07, WT14} in the context of time evolution (see also the review \cite[Section 4]{J14}). Let also mention here \cite{B85} dealing with the semiclassical counting function in the Born-Oppenheimer approximation (in a pseudo-differential context). Our strategy gives rise to a rather short proof (which does not even require approximations of the eigenfunctions) and displays a uniformity in the spectral estimates that implies the Weyl formula of Theorem \ref{theo.Weyl1}. In Section \ref{sec.neg}, we establish Theorem \ref{theo.Weyl2}. Note that the proof of Theorem~\ref{theo.Weyl2} does not follow from a reduction to the effective Hamiltonian but uses a variational argument as the one in \cite{CdV}. This argument is based on a decomposition of the operator via a rough partition of the unity and a separation of variables.

\section{The Robin Laplacian near the boundary}\label{sec.red-bnd}
\subsection{Reduction near the boundary via Agmon estimates}
The eigenfunctions (with negative eigenvalues) of the initial operator $\mathcal{L}_h$ are localized near the boundary since the Laplacian is non negative inside the domain. This localization is quantified by the following proposition (the proof of which is a direct adaptation of the case in dimension two, see \cite{HK-tams} and also \cite{H88}).
\begin{proposition}\label{thm:dec}
Let $\epsilon_{0}\in(0,1)$ and $\alpha\in(0,\sqrt{\epsilon_{0}})$. There exist constants $C>0$  and $h_0\in(0,1)$ such
that, for  $h\in(0,h_0)$,  if $u_h$ is a normalized eigenfunction of $\mathcal L_h$ with eigenvalue $\mu\leq-\epsilon_{0}h$, then,
\[\int_\Omega \left(|u_h(\x)|^2+h|\nabla u_h(\x)|^2\right)\exp\left(\frac{2\alpha\, {\rm dist}(\x,\Gamma)}{h^{\frac 1 2}}\right)\dx \x\leq C\,.\]
\end{proposition}
Given $\delta \in (0,\delta_0)$ (with $\delta_0>0$ small enough), we introduce the $\delta$-neighborhood of the boundary
\begin{equation}\label{eq.delta.neighbor}
\mathcal V_{\delta}=\{\x\in\Omega~:~{\rm dist}(\x,\Gamma)<\delta\}\,,
\end{equation}
and the quadratic form, defined on the variational space
\[V_{\delta}=\{u\in H^1(\mathcal V_{\delta})~:~u(\x)=0\,,\quad\mbox{ for all } \x\in\Omega \mbox{ such that } {\rm dist}(\x,\Gamma)=\delta\}\,,\]
by the formula
\[\forall u\in V_{\delta}\,,\qquad\mathcal{Q}_{h}^{\{\delta\}}(u)=\int_{\mathcal{V}_{\delta}}|h\nabla u|^2 \dx\x-h^{\frac{3}{2}}\int_{\Gamma}|u|^2\dx \Gamma\,.\]
Let us denote by $\mu^{\{\delta\}}_n(h)$ the $n$-th eigenvalue of the corresponding operator $\mathcal{L}_{h}^{\{\delta\}}$. It is then standard to deduce from the min-max principle and the Agmon estimates  of Proposition~\ref{thm:dec} the following proposition (see \cite{H88}).
\begin{proposition}\label{prop:red-bnd}
Let $\epsilon_{0}\in(0,1)$ and $\alpha\in(0,\sqrt{\epsilon_{0}})$.There exist constants $C>0$, $h_0\in(0,1)$ such that, for all $h\in(0,h_0)$, $\delta\in(0,\delta_{0})$, $n\geq 1$ such that $\mu_{n}(h)\leq-\epsilon_{0}h$,
\begin{equation}\label{complem}
\mu_n(h)\leq \mu^{\{\delta\}}_n(h)\leq \mu_n(h)+C\exp\left(-\alpha  \delta h^{-\frac 12} \right)\,.
\end{equation}
\end{proposition}

\subsection{Description of the boundary coordinates}\label{sec.bnd}
Let $\iota$ denote the embedding of $\Gamma$ in $\R^d$ and $g$ the induced metrics on $\Gamma$. $(\Gamma,g)$ is a $\mathcal{C}^3$ Riemmanian manifold, which we orientate according to the ambient space. Let us introduce the map $\Phi:\Gamma\times(0,\delta)\to\mathcal{V}_{\delta}$ defined by the formula
\begin{equation*}
\Phi(s,t)=\iota(s)+t\n(s)\,,
\end{equation*}
which we assume to be injective. The transformation $\Phi$ is a $\mathcal{C}^3$ diffeomorphism for $\delta\in(0,\delta_0)$ and $\delta_0$ is sufficiently small. The induced metrics on $\Gamma\times(0,\delta)$ is given by
\[G=g\circ (\mathsf{Id}-tL(s))^{2}+\dx t^2\,,\]
where $L(s)=-d\n_{s}$ is the second fondamental form of the boundary at $s$.

\subsection{The Robin Laplacian in boundary coordinates}\label{sec.Lbnd}
For all $u\in L^{2}(\mathcal V_{\delta_{0}})$, we define the pull-back function
\begin{equation}\label{eq:trans-st}
\widetilde u(s,t):= u(\Phi(s,t)).
\end{equation}
For all $u\in H^{1}(\mathcal{V}_{\delta_{0}})$, we have
\begin{equation}\label{eq:bc;n}
\int_{\mathcal{V}_{\delta_{0}}}|u|^{2}\dx\x=\int_{\Gamma\times(0,\delta_0)}|\widetilde u(s,t)|^{2}\,\tilde a\dx \Gamma \dx t\,,
\end{equation}
\begin{equation}\label{eq:bc;qf}
\int_{\mathcal{V}_{\delta_{0}}}|\nabla u|^{2}\dx\x= \int_{\Gamma\times(0,\delta_0)} \Big[\langle\nabla_{s} \widetilde u,\tilde g^{-1}\nabla_{s} \widetilde u\rangle +|\partial_{t}\widetilde u|^{2}\Big]\,\tilde a\dx \Gamma\dx t\,.
\end{equation}
where
\[\tilde g=\big(\mathsf{Id}-tL(s)\big)^2\,,\]
and $\tilde a(s,t)= |\tilde g(s,t)|^{\frac{1}{2}}$. Here $\langle\cdot,\cdot\rangle$ is the Euclidean scalar product and $\nabla_{s}$ is the differential on $\Gamma$ seen through the metrics $g$.

The operator $\mathcal L^{\{\delta\}}_h$ is expressed in $(s,t)$ coordinates as
\[\mathcal L^{\{\delta\}}_h=-h^2\tilde a^{-1}\nabla_s(\tilde a \tilde g^{-1}\nabla_s)-h^2\tilde a^{-1}\partial_t(\tilde a\partial_t)\,,\]
acting on $L^2(\tilde a\dx \Gamma\dx t)$. In these coordinates, the Robin condition becomes
\[h^2\partial_tu=-h^{\frac 32}u\quad{\rm on}\quad t=0\,.\]
We introduce, for $\delta \in (0,\delta_0)$,
\begin{equation}\label{eq:red-bnd}
\begin{aligned}
&\widetilde{\mathcal V}_\delta=\{(s,t)~:~s\in \Gamma~{\rm and}~0<t< \delta\}\,, \\ 
& \widetilde V_\delta =\{u\in H^1(\widetilde{\mathcal V_\delta})~:~u(s,\delta)=0\}\,,\\
&\widetilde{\mathcal D}_\delta=\{u\in H^2(\widetilde{\mathcal V_{\delta}})\cap \widetilde V_\delta ~:~\partial_tu(s,0)=-h^{-\frac{1}{2}}u(s,0)\}\,,\\
&\widetilde{\mathcal{Q}}_{h}^{\{\delta\}}(u)=\int_{\widetilde{\mathcal V_\delta}}\Big(h^2\langle\nabla_{s} u,\tilde g^{-1}\nabla_{s} u\rangle+|h\partial_tu|^2\Big)\tilde a\dx \Gamma\dx t-h^{\frac 32}\int_{\Gamma} |u(s,0)|^2 \dx \Gamma\,,\\
&\widetilde{\mathcal L}_h^{\{\delta\}}= -h^2\tilde a^{-1}\nabla_s(\tilde a \tilde g^{-1}\nabla_s)-h^2\tilde a^{-1}\partial_t(\tilde a\partial_t)\,.
\end{aligned}
\end{equation}
We now take \begin{equation}
\delta=  h^{\rho}\,,
\end{equation}
and write simply $\widetilde{\mathcal L}_h$ for $\widetilde{\mathcal L}_h^{\{\delta\}}$.
The operator $\widetilde{\mathcal L}_h$ with domain $\widetilde{\mathcal D}$ is the self-adjoint operator defined via the closed quadratic form $\widetilde{\mathcal V}_{\rho}\ni u\mapsto \widetilde{\mathcal{Q}}_{h}(u)$ by Friedrich's theorem.
\subsection{The rescaled operator}
We introduce the rescaling
\[(\sigma,\tau)=(s,h^{-\frac 12}t)\,,\]
the new semiclassical parameter $\hbar=h^{\frac 1 4}$ and the new weights
\begin{equation}\label{eq:Jac-a'}
\widehat a(\sigma,\tau)=\tilde a(\sigma,h^{\frac{1}{2}}\tau)\,,\qquad \widehat g(\sigma,\tau)=\tilde g(\sigma,h^{\frac{1}{2}}\tau)\,.
\end{equation}
We consider rather the operator
\begin{equation}\label{eq.Lh-hat}
\widehat{\mathcal{L}}_{\hbar}=h^{-1}\widetilde{\mathcal{L}}_{h}\,,
\end{equation}
acting on $L^2(\widehat a\dx\Gamma \dx\tau)$ and expressed in the coordinates $(\sigma,\tau)$. As in \eqref{eq:red-bnd}, we let
\begin{equation}\label{eq:dom-Lh-hat}
\begin{aligned}
&\widehat{\mathcal V}_{T }=\{(\sigma,\tau)~:~\sigma\in \Gamma~{\rm and}~0<\tau<
 T  \}\,, \\
&\widehat V_{T}=\{u\in H^1(\widehat{\mathcal V}_{T })~:~u(\sigma, T ) =0\}\,,\\
&\widehat{\mathcal D}_{T }=\{u\in H^2(\widehat{\mathcal V}_{T })\cap \widehat V_{T }~:~\partial_\tau u(\sigma,0)=-u(\sigma,0)\}\,,\\
&\widehat{\mathcal{Q}}_{\hbar}^{T }(u)=\int_{\widehat{\mathcal V}_{T }}\Big(\hbar^4\langle\nabla_{\sigma} u,\widehat g^{-1}\nabla_{\sigma} u\rangle+|\partial_\tau u|^2\Big)\widehat a \dx\Gamma \dx\tau-\int_\Gamma |u(\sigma,0)|^2\dx\Gamma\,,\\
&\widehat{\mathcal{L}}_\hbar^{T }=-\hbar^4 \widehat a^{-1}\nabla_\sigma(\widehat a \widehat g^{-1}\nabla_\sigma)-\widehat a^{-1}\partial_\tau\widehat a\partial_\tau\,.
\end{aligned}
\end{equation}
\begin{notation}
In what follows, we let $T=\hbar^{-1}$ (or equivalently $\rho=\frac{1}{4}$) and write $\widehat{\mathcal{Q}}_{\hbar}$ for $\widehat{\mathcal{Q}}_{\hbar}^{T }$.
\end{notation}

\section{A variational Born-Oppenheimer reduction}\label{sec.eff}
The aim of this section is to prove the following result that implies Theorem \ref{theo.Weyl1}.
\begin{theorem}\label{theo.eff}
For $\eps_{0}\in(0,1)$, $h>0$, we let\[\mathcal{N}_{\epsilon_{0}, h}=\{n\in\mathbb{N}^* : \mu_{n}(h)\leq-\eps_{0}h\}\,.\]
There exist positive constants $h_{0}, C_{+}, C_{-}$ such that, for all $h\in(0,h_{0})$ and $n\in\mathcal{N}_{\eps_{0}, h}$,
 \begin{equation}\label{eq:App.eh}
\mu^-_{n}(h)\leq\mu_{n}(h)\leq \mu^{+}_{n}(h)\,,\end{equation}
where $\mu^{\pm}_{n}(h)$ is the $n$-th eigenvalue of $\mathcal{L}^{\mathsf{eff}, \pm}_{h}$ defined by
 \[\mathcal{L}^{\mathsf{eff},+}_{h}=-h+(1+C_{+}h^{\frac{1}{2}})h^2\mathcal{L}^{\Gamma}-\kappa h^{\frac{3}{2}}+C_{+}h^2\,,\]
and
 \[\mathcal{L}^{\mathsf{eff},-}_{h}=-h+(1-C_{-}h^{\frac{1}{2}})h^2 \mathcal{L}^{\Gamma}-\kappa h^{\frac{3}{2}}-C_{-}h^{2}\,.\]
\end{theorem}

\subsection{The corrected Feshbach projection}
Let us introduce 
\[\mathcal{H}_{\kappa(\sigma),\hbar} = \mathcal H^{\{T\}}_{B}\,,\]
with 
\[B=h^{\frac{1}{2}}\kappa(\sigma)=\hbar^2\kappa(\sigma)\] 
and where $\mathcal H^{\{T\}}_{B}$ is defined in \eqref{eq:H0b}. We introduce for $\sigma \in \Gamma$ the Feshbach projection $\Pi_\sigma$  on the normalized groundstate of  $ \mathcal{H}_{\kappa(\sigma),\hbar}$, denoted by $v_{\kappa(\sigma), \hbar}$,
\[
\Pi_\sigma  \psi=\langle\psi, v_{\kappa(\sigma), \hbar}\rangle_{L^2((0,T),(1-B\tau) \dx\tau)} v_{\kappa(\sigma), \hbar}\,.
\]
We also let
\[\Pi_\sigma^\perp=\mathsf{Id}-\Pi_\sigma \]
and
\begin{align}\label{deff}
&f(\sigma)=\langle\psi, v_{\kappa(\sigma), \hbar}\rangle_{L^2((0,T),(1-B\tau) \dx\tau)}, \\
\label{defR}
&R_{\hbar}(\sigma)=\|\nabla_{\sigma} v_{\kappa(\sigma), \hbar}\|^2_{L^2((0,T),\,(1-B\tau) \dx\tau)}\,,
\end{align}
The quantity $R_{\hbar}$ is sometimes called \enquote{Born-Oppenheimer correction}. It measures the commutation defect between $\nabla_{\sigma}$ and $\Pi_{\sigma}$.
\begin{remark}
\rm In a first approximation, one could try to use the projection on $v_{0,\hbar}$, but one would lose the uniformity in our estimates. Note that the idea to consider a corrected Feshbach projection appears in many different contexts: WKB analysis (see for instance \cite[Sections 2.4~\&~3.2]{BHR}, \cite{HK-tams} and \cite{HKR15}), norm resolvent convergence (see for instance \cite[Section 4.2]{KR13}) or space/time adiabatic limits (see \cite[Chapter 3]{T03}).
\end{remark}

\subsection{Approximation of the metrics}
In this section, we introduce an approximated quadratic form by approximating first the metrics. For that purpose, let us introduce the approximation of the weight:
\[\tilde m(s, t)=1-t\kappa(s)\,,\qquad\kappa(s)=\mathsf{Tr}\, L(s)\,.\]
We have
\[|\tilde a(s,t)-\tilde m(s,t)|\leq Ct^2\,.\]
Let us now state two elementary lemmas.
\begin{lemma}\label{lem.dtau}
We have the estimate, for all $\psi\in \widehat{V}_{T}$,
\begin{multline*}
\left|\int_{\widehat{\mathcal{V}}_{T }}|\partial_\tau \psi|^2\widehat a \dx\Gamma \dx\tau-\int_{\widehat{\mathcal{V}}_{T }}|\partial_\tau \psi|^2 \widehat m \dx\Gamma \dx\tau\right|\\
\leq C\hbar^4\int_\Gamma |f(\sigma)|^2\dx \Gamma+C\hbar^2\int_{\widehat{\mathcal{V}}_{T }}|\partial_{\tau}\Pi^\perp_{\sigma}\psi|^2\dx\Gamma\dx\tau\,,
\end{multline*}
where $\widehat m(\sigma,\tau)=\tilde m(\sigma,\h^2\tau)$.
\end{lemma}
\begin{proof}
We have
\[
\left|\int_{\widehat{\mathcal V}_{T }}|\partial_\tau \psi|^2\widehat a \dx\Gamma \dx\tau-\int_{\widehat{\mathcal V}_{T }}|\partial_\tau \psi|^2 \widehat m \dx\Gamma \dx\tau\right|\leq C\hbar^4\int_{\widehat{\mathcal V}_{T }}\tau^2|\partial_{\tau}\psi|^2\dx\Gamma\dx\tau \,.
\]
Then, we use an orthogonal decomposition to get
\begin{align*}
&\left|\int_{\widehat{\mathcal V}_{T }}|\partial_\tau \psi|^2\widehat a \dx\Gamma \dx\tau-\int_{\widehat{\mathcal V}_{T }}|\partial_\tau \psi|^2 \widehat m \dx\Gamma \dx\tau\right|\\
&\leq \tilde C\hbar^4\left(\int_{\widehat{\mathcal V_{T }}}\tau^2|\partial_{\tau}\Pi_{\sigma}\psi|^2\dx\Gamma\dx\tau+\int_{\widehat{\mathcal V}_{T }}\tau^2|\partial_{\tau}\Pi^\perp_{\sigma}\psi|^2\dx\Gamma\dx\tau\right)\\
&\leq  C\hbar^4\int_\Gamma |f(\sigma)|^2\left(\int_0^T \tau^2 |\partial_\tau v_{\kappa(\sigma), \hbar}|^2 \dx\tau \right) \dx \Gamma+C\hbar^2\int_{\widehat{\mathcal V}_{T }}|\partial_{\tau}\Pi^\perp_{\sigma}\psi|^2\dx\Gamma\dx\tau\,,
\end{align*}
where we have used that $T=\hbar^{-1}$ for the orthogonal component. The result then follows from the Agmon estimates in one dimension (Proposition \ref{prop.AgmonuBT}).

\end{proof}
\begin{lemma}\label{lem.dsigma}
We have the estimate, for all $\psi\in \widehat{V}_{T}$,
\begin{multline*}
\left|\int_{\widehat{\mathcal{V}}_{T }}\langle\nabla_{\sigma} \psi,\widehat g^{-1}\nabla_{\sigma} \psi\rangle\widehat a\dx\Gamma \dx\tau-\int_{\widehat{\mathcal V}_{T }}\langle\nabla_{\sigma} \psi, \nabla_{\sigma} \psi\rangle\widehat m\dx\Gamma \dx\tau\right|\\
\leq C  \int_\Gamma \left( \hbar^2\|\nabla_\sigma f(\sigma)\|^2 + \hbar R_\hbar(\sigma)|f(\sigma)|^2 \right) \dx \Gamma + C \hbar \int_{\widehat{\mathcal{V}}_{T }}\|\nabla_{\sigma}\Pi^\perp_{\sigma}\psi\|^2\dx\Gamma\dx\tau\,\,.
\end{multline*}

\end{lemma}
\begin{proof}
First, we write
\begin{align*}
&\left|\int_{\widehat{\mathcal V}_{T }}\langle\nabla_{\sigma} \psi,\widehat g^{-1}\nabla_{\sigma} \psi\rangle\widehat a\dx\Gamma \dx\tau-\int_{\widehat{\mathcal V}_{T }}\langle\nabla_{\sigma} \psi, \nabla_{\sigma} \psi\rangle\widehat m\dx\Gamma \dx\tau\right|\\
&\leq \int_{\widehat{\mathcal V}_{T }} \|\nabla_{\sigma} \psi\|^2|\widehat a-\widehat m|\dx\Gamma \dx\tau
+\int_{\widehat{\mathcal V}_{T }}|\langle\nabla_{\sigma} \psi,(\widehat  g^{-1} - \mathsf{Id})\nabla_{\sigma} \psi\rangle|\widehat a\dx\Gamma \dx\tau\\
&\leq C  \int_{\widehat{\mathcal V}_{T }} (\hbar^4 \tau^2+\hbar^2 \tau) \|\nabla_{\sigma} \psi\|^2\dx\Gamma \dx\tau\,.
\end{align*}
Then, by an orthogonal decomposition, we get
\begin{align*}
&\left|\int_{\widehat{\mathcal V}_{T }}\langle\nabla_{\sigma} \psi,\widehat g^{-1}\nabla_{\sigma} \psi\rangle\widehat a\dx\Gamma \dx\tau-\int_{\widehat{\mathcal V}_{T }}\langle\nabla_{\sigma} \psi, \nabla_{\sigma} \psi\rangle\widehat m\dx\Gamma \dx\tau\right|\\
&\leq C  \int_{\widehat{\mathcal V}_{T }} (\hbar^4\tau^2 +\hbar^2 \tau) \|\nabla_{\sigma} \Pi_{\sigma} \psi\|^2\dx\Gamma \dx\tau + C \hbar \int_{\widehat{\mathcal V}_{T}} \|\nabla_{\sigma} \Pi_{\sigma}^\perp \psi\|^2\dx\Gamma \dx\tau\,,
\end{align*}
where we used $T=\hbar^{-1}$ on the orthogonal part.

Finally, we use the naive inequality
\[
\|\nabla_{\sigma} \Pi_{\sigma} \psi\|^2 \leq 2 \left( \|\nabla_{\sigma} f(\sigma)\|^2 |v_{\kappa(\sigma),\hbar}|^2 + \|\nabla_{\sigma} v_{\kappa(\sigma),\hbar}\|^2 |f(\sigma)|^2 \right)\,,
\]
and the conclusion again follows from Agmon estimates.
\end{proof}

Let us now introduce the approximated quadratic form
\begin{equation}\label{eq.Qapp}
 \widehat{\mathcal{Q}}_\hbar^{\mathsf{app}}(\psi) = \int_{\widehat{\mathcal{V}}_{T }}\Big(\hbar^4\|\nabla_{\sigma} \psi\|^2
 +|\partial_\tau \psi|^2\Big)\widehat m \dx\Gamma \dx\tau
 -\int_\Gamma |\psi(\sigma,0)|^2\dx\Gamma\,.
\end{equation}
The sense of this approximation is quantified by the following lemma (that is a consequence of Lemmas \ref{lem.dtau} and \ref{lem.dsigma}).
\begin{lemma}\label{lem.approx}
We have, for all $\psi\in\widehat{V}_{T}$,
\begin{align*}
&\left|\widehat{\mathcal{Q}}_{\hbar}(\psi)-\widehat{\mathcal{Q}}_{\hbar}^{\mathsf{app} }(\psi)\right|\\
&\leq  C\hbar^4\int_\Gamma |f(\sigma)|^2\dx \sigma+C\hbar^2\int_{\widehat{\mathcal{V}}_{T }}|\partial_{\tau}\Pi^\perp_{\sigma}\psi|^2\dx\Gamma\dx\tau\\
&+C  \int_\Gamma \left( \hbar^6\|\nabla_\sigma f(\sigma)\|^2 + \hbar^5 R_\hbar(\sigma)|f(\sigma)|^2 \right) \dx \Gamma + C \hbar^5 \int_{\widehat{\mathcal{V}}_{T }}\|\nabla_{\sigma}\Pi^\perp_{\sigma}\psi\|^2\dx\Gamma\dx\tau\,.
\end{align*}
\end{lemma}

\subsection{Upper bound}
The following proposition provides an upper bound of the quadratic form on a subspace.
\begin{proposition}\label{upper_bound}
There exist $C>0$, $\hbar_0>0$ such that, for all $\psi \in \widehat{\mathcal D}_{T }$ and $\hbar \in (0,\hbar_0)$, we have
\begin{multline*}
\widehat{\mathcal{Q}}_{\hbar}(\Pi_{\sigma} \psi) \leq 
\int_\Gamma \hbar^4(1+C\hbar^2)\|\nabla_\sigma f(\sigma)\|^2  \dx \Gamma \\
 + \int_\Gamma \left( \lambda_1(\mathcal{H}_{\kappa(\sigma),\hbar}) +C\hbar^4 + \hbar^4(1+C\hbar)R_\hbar(\sigma) \right)|f(\sigma)|^2 \dx \Gamma\,. 
\end{multline*}
\end{proposition}

\begin{proof}
First, we use Lemma \ref{lem.approx} (note that the orthogonal projections disappear). Then, we are reduced to estimates on the approximated quadratic form.
By writing $ \Pi_\sigma \psi = f(\sigma) v_{\kappa(\sigma),\hbar} $ and considering the derivative of this product, we get
\begin{align*}
&\widehat{\mathcal{Q}}_{\hbar}^{\mathsf{app} }( \Pi_\sigma \psi)=\int_{\widehat{\mathcal V}_{T}}\Big(\hbar^4\|\nabla_{\sigma} \Pi_{\sigma} \psi\|^2+|\partial_\tau \Pi_{\sigma} \psi|^2\Big)\widehat m \dx\Gamma \dx\tau-\int_\Gamma |\Pi_{\sigma} \psi(\sigma,0)|^2 \dx\Gamma \\
&= \hbar^4 \int_\Gamma \left( \|\nabla_\sigma f(\sigma)\|^2 + \big(R_\hbar(\sigma)+q_{\kappa(\sigma),\hbar}(v_{\kappa(\sigma),\hbar})\big) |f(\sigma)|^2 \right)  \dx \Gamma  \\
&+ 2 \hbar^4 \int_\Gamma f(\sigma) \left\langle \nabla_\sigma f(\sigma), \int_0^T v_{\kappa(\sigma),\hbar} \nabla_\sigma v_{\kappa(\sigma),\hbar} \widehat{m} \dx \tau \right\rangle  \dx \Gamma\,.
\end{align*}  
where $q_{\kappa(\sigma),\hbar}$ is the quadratic form associated with $\mathcal{H}_{\kappa(\sigma),\hbar}$. 
By definition, we have 
\[q_{\kappa(\sigma),\hbar}(v_{\kappa(\sigma),\hbar})=\lambda_1(\mathcal{H}_{\kappa(\sigma),\hbar})\,.\]
Then we notice from the normalization of $v_{\kappa(\sigma),\hbar}$ that
\[\nabla_\sigma \left( \int_0^T |v_{\kappa(\sigma),\hbar}|^2   \widehat{m} \dx \tau \right) = 0\,,\]
and since $\nabla_\sigma B=\hbar^2 \nabla_\sigma \kappa(\sigma)$, we have
\begin{equation*}
\int_0^T v_{\kappa(\sigma),\hbar} \nabla_\sigma v_{\kappa(\sigma),\hbar}  \widehat{m} \dx \tau = \mathcal{O}(\hbar^2)\,.
\end{equation*}
This implies the estimate:
\begin{multline}\label{eq.quasi-orth}
\left|\hbar^4 \int_\Gamma  f(\sigma) \left\langle \nabla_\sigma f(\sigma), \int_0^T v_{\kappa(\sigma),\hbar} \nabla_\sigma v_{\kappa(\sigma),\hbar}  \widehat{m} \dx \tau \right\rangle \dx \Gamma\right| \\
\leq C \hbar^6 \int_\Gamma \left(|f(\sigma)|^2+\|\nabla_{\sigma} f(\sigma)\|^2\right) \dx \Gamma\,,
\end{multline}
and the conclusion follows.
\end{proof}

\subsection{Lower bound}

Let us now establish the following lower bound of the quadratic form.
\begin{proposition}\label{lower_bound}
There exist $C>0$, $\hbar_0>0$ such that, for all $\psi \in \widehat{\mathcal D}_{T }$ and $\hbar \in (0,\hbar_0)$, we have
\begin{multline*}
\widehat{\mathcal{Q}}_{\hbar}(\psi) \\
\geq \int_\Gamma \left( \hbar^4(1-C\hbar^2)\|\nabla_\sigma f(\sigma)\|^2 + \left( \lambda_1(\mathcal{H}_{\kappa(\sigma),\hbar}) -C(\hbar^4+\hbar^{2}R_\hbar(\sigma) \right) |f(\sigma)|^2 \right) \dx \Gamma \\
+ \int_\Gamma  \hbar^4(1-C\hbar)\|\nabla_\sigma \Pi_\sigma^\perp \psi\|^2_{L^2( \widehat{m}\dx \tau)} \dx \Gamma  \\
+ \int_\Gamma \big(  (1-C\hbar^2)  \lambda_2(\mathcal{H}_{\kappa(\sigma),\hbar}) - C (\hbar^6+\hbar^2 R_\hbar(\sigma)) \big) \|\Pi_\sigma^\perp \psi\|^2_{L^2( \widehat{m}\dx \tau)}  \dx \Gamma\,.
\end{multline*}
\end{proposition}

\begin{proof}
The proof will be done in a few steps.
\begin{enumerate}[i.]
\item First, we use Lemma \ref{lem.approx} to write
\begin{multline}\label{eq.lb0}
 \widehat{\mathcal{Q}}_{\hbar}(\psi) \geq 
 \int_{\widehat{\mathcal V}_{T }}\hbar^4\|\nabla_{\sigma} \psi\|^2 \widehat m \dx\Gamma \dx\tau 
 + \int_\Gamma q_{\kappa(\sigma),\hbar} (\psi) \dx\Gamma \\
 - C \int_\Gamma \left( \hbar^6\|\nabla_\sigma f(\sigma)\|^2 + \hbar^4(1+ \hbar R_\hbar(\sigma))|f(\sigma)|^2 \right) \dx \Gamma \\
 - C \hbar^5 \int_{\widehat{\mathcal{V}}_{T }}\|\nabla_{\sigma}\Pi^\perp_{\sigma}\psi\|^2 \widehat m \dx\Gamma\dx\tau
 - C \hbar^2\int_{\widehat{\mathcal{V}}_{T }}|\partial_{\tau}\Pi^\perp_{\sigma}\psi|^2 \widehat m \dx\Gamma\dx\tau\,.
\end{multline}

\item 
On one hand, we get, by using an orthogonal decomposition, for each $\sigma\in\Gamma$,
\[q_{\kappa(\sigma),\hbar} (\psi)=q_{\kappa(\sigma),\hbar} (\Pi_\sigma \psi)+q_{\kappa(\sigma),\hbar} (\Pi^\perp_\sigma \psi)\,.\]
Then, we get, by using the min-max principle,
\begin{align}
 \nonumber&\int_\Gamma q_{\kappa(\sigma),\hbar} (\psi) \dx\Gamma
  - C \hbar^2\int_{\widehat{\mathcal{V}}_{T }}|\partial_{\tau}\Pi^\perp_{\sigma}\psi|^2 \widehat m \dx\Gamma\dx\tau \\
 \label{eq.lb1}&\geq  \int_\Gamma q_{\kappa(\sigma),\hbar} (\Pi_\sigma \psi) \dx\Gamma + (1-C \hbar^2)\int_\Gamma q_{\kappa(\sigma),\hbar} (\Pi_\sigma^\perp \psi) \dx\Gamma \\
 \nonumber&\geq \int_\Gamma \left( \lambda_1(\mathcal{H}_{\kappa(\sigma),\hbar}) |f(\sigma)|^2 + (1-C \hbar^2) \lambda_2(\mathcal{H}_{\kappa(\sigma),\hbar}) \|\Pi_\sigma^\perp \psi\|^2_{L^2( \widehat{m}\dx\tau)} \right) \dx\Gamma\,.
\end{align}
On the other hand, we also have
\begin{equation}\label{eq.lb2}
\|\nabla_{\sigma} \psi\|^2_{L^2(\widehat{m}\dx\tau)}
   =\|\Pi_\sigma \nabla_{\sigma} \psi\|_{L^2(\widehat{m}\dx\tau)}^2
   +\|\Pi_\sigma^\perp \nabla_{\sigma} \psi\|^2_{L^2(\widehat{m}\dx\tau)}\,.
\end{equation}
\item Then, we estimate the commutator:
\begin{multline*}
 \left[ \nabla_\sigma , \Pi_\sigma \right] \psi 
  = \langle\psi, \nabla_\sigma v_{\kappa(\sigma), \hbar}\rangle_{L^2( \widehat{m} \dx\tau)} v_{\kappa(\sigma), \hbar}
  + \langle\psi, v_{\kappa(\sigma), \hbar}\rangle_{L^2( \widehat{m})\dx\tau} \nabla_\sigma v_{\kappa(\sigma), \hbar} \\
  - \hbar^2 \nabla_\sigma \kappa(\sigma) \left( \int_0^T \psi v_{\kappa(\sigma), \hbar} \tau \dx \tau \right) v_{\kappa(\sigma), \hbar}\,.
\end{multline*}
We get, thanks to the Cauchy-Schwarz inequality and Agmon estimates (see Proposition \ref{prop.AgmonuBT}),
\begin{equation}\label{eq.comestim}
 \left\|\left[ \Pi_\sigma, \nabla_\sigma \right] \psi \right\|_{L^2( \widehat{m} \dx\tau)}
 \leq \left( 2 R_\hbar(\sigma)^{\frac{1}{2}} +C \hbar^2 \right) \|\psi\|_{L^2( \widehat{m} \dx\tau)}\,.
\end{equation}
Then, we write 
\begin{equation}\label{eq.commut}
\Pi_\sigma \nabla_{\sigma} \psi = \nabla_\sigma f(\sigma) v_{\kappa(\sigma), \hbar}
  + f(\sigma) \nabla_\sigma v_{\kappa(\sigma),\hbar} 
  + \left[ \Pi_\sigma, \nabla_\sigma \right] \psi \,.\
\end{equation}
Let us recall the following classical inequality:
\[\forall \, a,b \in \C^{n-1},  \forall \, \varepsilon \in (0,1), \, \|a+b\|^2 \geq (1-\varepsilon)\|a\|^2 - \varepsilon^{-1}\|b\|^2\,.\]
We take $\varepsilon=\hbar^2$, $a= \nabla_\sigma f(\sigma) v_{\kappa(\sigma), \hbar}$ and $b= f(\sigma) \nabla_\sigma v_{\kappa(\sigma),\hbar} + \left[ \Pi_\sigma, \nabla_\sigma \right] \psi $. We get, from  \eqref{eq.comestim} and \eqref{eq.commut},
\begin{multline}\label{eq.Pinab}
 \int_0^T\|\Pi_\sigma\nabla_{\sigma} \psi\|^2  \widehat{m} \dx\tau
 \geq (1-\hbar^2) \|\nabla_\sigma f(\sigma)\|^2 \\
 -C\hbar^{-2} \big(R_\hbar(\sigma)+\mathcal{O}(\hbar^4)\big) \big(|f(\sigma)|^2 + \| \Pi_\sigma^\perp \psi \|^2_{L^2(\widehat{m}\dx\tau)} \big)\,.
\end{multline}
In the same way, we get 
\begin{multline}\label{eq.Pipnab}
 \int_0^T\|\Pi_\sigma^\perp\nabla_{\sigma} \psi\|^2  \widehat{m} \dx\tau
 \geq (1-\hbar^2) \|\nabla_\sigma \Pi_\sigma^\perp \psi \|_{L^2(\widehat{m}\dx\tau)}^2 \\
 -C\hbar^{-2} \big(R_\hbar(\sigma)+\mathcal{O}(\hbar^4)\big) \big(|f(\sigma)|^2 + \| \Pi_\sigma^\perp \psi \|^2_{L^2(\widehat{m}\dx\tau)} \big)\,. 
\end{multline} 
\item Now we use \eqref{eq.lb0}, \eqref{eq.lb1}, \eqref{eq.lb2} and the estimates \eqref{eq.Pinab}, \eqref{eq.Pipnab} and the conclusion follows.
\end{enumerate}
\end{proof}

\subsection{Derivation of the effective Hamiltonians}
We can now end the proof of Theorem \ref{theo.eff}.
\begin{enumerate}[i.]

\item We apply Proposition \ref{lem:H0b;l} to get
\[\lambda_1(\mathcal{H}_{\kappa(\sigma),\hbar})=-1- \kappa(\sigma)\h^2+\mathcal{O}(\h^4)\,,\]
and we use Lemmas \ref{lem:1DL}, \ref{lem:H0b} to deduce that there exist positive constants $\h_0$ and $C$ such that, for all $\h \in (0,\h_0)$,
\[\lambda_2(\mathcal{H}_{\kappa(\sigma),\hbar})\geq-C\h\geq-\frac{\eps_0}{2}\,.\] 
Then we notice, thanks to Lemma \ref{lem.h12}, that the Born-Oppenheimer correction satisfies $R_\h(\sigma)=\mathcal{O}(\h^4)$.

\item As a consequence of Proposition \ref{upper_bound}, there exists $C_+>0$ such that, for all $\psi\in\widehat{\mathcal{D}}_T$ and $\h$ small enough, 
\[\widehat{\mathcal{Q}}_{\hbar}(\Pi_{\sigma} \psi) \leq \widehat{\mathcal{Q}}^{\mathsf{eff},+}_\h(f)\,,\]
where, for all $f \in H^1(\Gamma)\,,$
\begin{equation*}
 \widehat{\mathcal{Q}}^{\mathsf{eff},+}_\h(f)
 =\int_\Gamma \left( \hbar^4(1+C_+\hbar^2)\|\nabla_\sigma f\|^2 
 + \left( -1-\kappa(\sigma)\h^2 +C_+\hbar^4 \right)|f|^2 \right)\dx \Gamma\,.
\end{equation*}
For $n\geq1$, let
\[G_{n,\h}=\left\{fv_{\kappa(\sigma),\h} \in \widehat{\mathcal{D}}_T\,:\,f\in F_{n,\h}\right\}\,,\]
 where $F_{n,\h}$ is the subspace of $H^1(\Gamma)$ spanned by the eigenvalues $\left(\widehat\mu_k^{\mathsf{eff},+}(\h)\right)_{1\leq k\leq n}$ of the associated operator $\widehat{\mathcal{L}}^{\mathsf{eff},+}_\h$. We have $\dim G_{n,\h}=n$ and, for all $\psi\in G_{n,\h}$,
 \[\widehat{\mathcal{Q}}_{\hbar}(\psi) \leq \widehat\mu_n^{\mathsf{eff},+}(\h)\|\psi\|^2_{L^2(\widehat a \dx\Gamma\dx\tau)}\,,\]
so that, by application of the min-max principle,
\[\widehat\mu_n(\h)\leq \widehat\mu_n^{\mathsf{eff},+}(\h)\,.\]

\item For $\eps_0 \in (0,1)$, thanks to Proposition \ref{lower_bound}, there exists $C_->0$ such that, for all $\psi \in\widehat{\mathcal{D}}_T$ and $\h$ small enough, 
\[\widehat{\mathcal{Q}}_{\h}(\psi) \geq \widehat{\mathcal{Q}}^{\mathsf{eff},-}_\h(f)-\frac{\eps_0}{2}\|\Pi_\sigma^\perp\psi\|^2_{L^2(\widehat m\dx\Gamma\dx\tau)}\,,\]
where, for all $f \in H^1(\Gamma)$,
\begin{equation*}
 \widehat{\mathcal{Q}}^{\mathsf{eff},-}_\h(f)
 =\int_\Gamma \left( \hbar^4(1-C_-\h^2)\|\nabla_\sigma f\|^2 
 + \left( -1-\kappa(\sigma)\h^2 -C_-\hbar^4 \right)|f|^2 \right)\dx \Gamma\,.
\end{equation*}
We consider the quadratic form defined, for $(f,\varphi)\in H^1(\Gamma)\times\widehat V_T$, by
\[\widehat{\mathcal{Q}}_\h^{\mathsf{tens}}(f,\varphi)=\widehat{\mathcal{Q}}^{\mathsf{eff},-}_\h(f)-\frac{\eps_0}{2}\|\varphi\|^2_{L^2(\widehat m\dx\Gamma\dx\tau)}\,.\]
By application of the min-max principle (see also \cite[Chapter 13]{Ray}), we have the comparison of the Rayleigh quotients:
\[\widehat\mu_n(\h)\geq \widehat\mu_n^{\mathsf{tens}}(\h)\,.\]
Note that the spectrum of $\widehat{\mathcal{L}}^\mathsf{tens}_{\h}$ lying below $-\eps_{0}$ is discrete and coincides with the spectrum of $\widehat{\mathcal{L}}_\h^{\mathsf{eff},-}$.
Then, for all $n\in\mathcal{N}_{\eps_{0}, h}$, $\widehat\mu_n^{\mathsf{tens}}(\h)$ is the $n$-th eigenvalue of $\widehat{\mathcal{L}}^\mathsf{tens}_{\h}$ and its satisfies $\widehat\mu_n^{\mathsf{tens}}(\h)=\widehat{\mu}^{\mathsf{eff},-}_{n}(\h)$.
\end{enumerate}

\section{Asymptotic counting formula for the non positive eigenvalues}\label{sec.neg}
This section is devoted to the proof of Theorem \ref{theo.Weyl2}. For that purpose we prove an upper bound in Proposition \ref{prop.ub} and a lower bound in Proposition \ref{prop.lb}.
\begin{proposition}\label{prop.ub}
There exist $C, h_{0}>0$ such that for all $h\in(0,h_{0})$,
\[\mathsf{N}\left(\mathcal{L}_{h},0\right)\leq (1+o(1))\mathsf{N}\left(h\mathcal{L}^\Gamma, 1\right)\,.\]
\end{proposition}
\begin{proof}
Consider a quadratic partition of the unity $(\chi_{j,h})_{j=1,2}$ in $\overline{\Omega}$ satisfying
\[\sum_{j=1}^2\chi_{j,h}^2=1\,,\quad \sum_{j=1}^2|\nabla\chi_{j,h}|^2\leq C h^{-2\rho}\,,\]
and
\[{\supp}\chi_{1,T}\subset \{{\dist}(x,\partial\Omega)< h^{\rho}\}\,.\]
For all $u\in H^1(\Omega)$, the following classical localization formula holds
\begin{equation}\label{eq:partition}
\begin{aligned}
\mathcal Q_h(u)&=\mathcal Q_h(\chi_{1,h}u)+\mathcal Q_h(\chi_{2,h}u)-h^2\sum_{j=1}^2\big\|u\nabla\chi_{j,h}\|^2 \\
&\geq \mathcal Q_h(\chi_{1,h}u)+\mathcal Q_h(\chi_{2,h}u)-Ch^{2-2\rho}\|u\|^2\,.
\end{aligned}
\end{equation}
Now, we estimate $\mathcal Q_h(\chi_{1,h}u)$ by using the boundary coordinates (see Section \ref{sec.Lbnd} and especially \eqref{eq:red-bnd}) and a rough Taylor expansion of the metrics:
\[\mathcal Q_h(\chi_{1,h}u)\geq (1-Ch^{\rho})\widetilde{\mathcal{Q}}_{h}^\mathsf{\mathsf{tens}}(\widetilde{\chi_{1,h}u})\,,\]
where
\[\widetilde{\mathcal{Q}}^\mathsf{tens}_{h}(v)=\int_{\widetilde{\mathcal V_\delta}}\Big(h^2\langle\nabla_{s} v,\nabla_{s} v\rangle+|h\partial_t v|^2\Big)\dx \Gamma\dx t-h^{\frac 32}\int_{\Gamma} |v(s,0)|^2 \dx \Gamma\,.\]
We deduce that
\[\mathcal Q_h(u)\geq  (1-Ch^{\rho})\widetilde{\mathcal{Q}}_{h}^\mathsf{\mathsf{tens}}(\widetilde{\chi_{1,h}u})+\mathcal Q_h(\chi_{2,h}u)-Ch^{2-2\rho}\|u\|^2\,.\]
Then, thanks to the min-max principle (see \cite{CdV}), we get:
\begin{equation}\label{eq.N+N}
\mathsf{N}\left(\mathcal{L}_{h},0\right)\leq \mathsf{N}\left(\widetilde{\mathcal{L}}^\mathsf{tens}_{h}, Ch^{2-2\rho}\right)+\mathsf{N}\left(-h^2\Delta^\mathsf{Dir}, Ch^{2-2\rho}\right)\,.
\end{equation}
Then, by using the usual Weyl formula for the Dirichlet Laplacian, we get
\begin{equation}\label{eq.Ndir}
\mathsf{N}\left(-h^2\Delta^\mathsf{Dir}, Ch^{2-2\rho}\right)\leq C h^{-d\rho}\,,
\end{equation}
and it remains to analyze $\mathsf{N}\left(\widetilde{\mathcal{L}}^\mathsf{tens}_{h}, Ch^{2-2\rho}\right)$. The operator $\widetilde{\mathcal{L}}^\mathsf{tens}_{h}$ is in a tensorial form and it has a Hilbertian decomposition by using the Hilbertian basis of the eigenfunctions of the transverse  Robin Laplacian. Let us describe the spectrum of the transverse operator and show that only its first eigenvalue contributes to the spectrum of $\widetilde{\mathcal{L}}^\mathsf{tens}_{h}$ below $Ch^{2-2\rho}$.

We know that the second eigenvalue of $h^2D^2_{t}$, acting on $L^2((0,h^{\rho}),\dx t)$, with Robin condition at $0$ and Dirichlet condition at $h^{\rho}$ is of order $h^{2\rho}$ (see Lemma \ref{lem:1DL*}, with $T=h^{\rho-\frac{1}{2}}$). Since $\rho\in\left(0,\frac{1}{2}\right)$, we get $h^{2\rho}\gg h^{2-2\rho}$ and thus we have only to consider the first transverse eigenvalue whose asymptotic expansion is $-h+\mathcal{O}(h^\infty)$. We get
\begin{equation}\label{eq.NL}
\mathsf{N}\left(\widetilde{\mathcal{L}}^\mathsf{tens}_{h}, Ch^{2-2\rho}\right)\leq \mathsf{N}\left(h\mathcal{L}^\Gamma, 1+\tilde Ch^{1-2\rho}\right)\underset{h\to 0}{\sim} \mathsf{N}\left(h\mathcal{L}^\Gamma, 1\right)\,.
\end{equation}
We deduce the upper bound by combining \eqref{eq.N+N}, \eqref{eq.Ndir}, \eqref{eq.NL}, \eqref{eq.Weyl-stand} and taking $\rho$ small enough.
\end{proof}
\begin{proposition}\label{prop.lb}
There exist $C, h_{0}>0$ such that for all $h\in(0,h_{0})$,
\[\mathsf{N}\left(\mathcal{L}_{h},0\right)\geq (1+o(1))\mathsf{N}\left(h\mathcal{L}^\Gamma, 1\right)\,.\]
\end{proposition}
\begin{proof}
To find the lower bound, we just have to bound the quadratic form $\mathcal{Q}_{h}$ on an appropriate subspace. We consider $\rho\in\left(0,\frac{1}{2}\right)$. We first notice that, for $u$ such that $\supp u\subset \widetilde{\mathcal V}_{h^\rho}$,
\[\mathcal Q_h(u)\leq (1+Ch^{\rho})\widetilde{\mathcal{Q}}_{h}^\mathsf{\mathsf{tens}}(\widetilde{u})\,.\]
We apply this inequality to the space spanned by functions in the form $\widetilde{u}(s,t)=f_{h, n}(s)u_{h}(t)$ where the $f_{h,n}$ are the eigenfunctions of $h^2\mathcal{L}^\Gamma+\lambda(h)$ associated with non positive eigenvalues and $u_{h}$ is the first eigenfunction of the transverse Robin Laplacian with eigenvalue $\lambda(h)=-h+\mathcal{O}(h^{\infty})$. The conclusion again follows from the min-max principle and the fact that  $\mathsf{N}\left(h\mathcal{L}^\Gamma, 1+\mathcal{O}(h^\infty)\right)\underset{h\to 0}{\sim}\mathsf{N}\left(h\mathcal{L}^\Gamma, 1\right)$.

\end{proof}

\appendix
\section{Reminders about Robin Laplacians in one dimension}
The aim of this section is to recall a few spectral properties related to the Robin Laplacian in dimension one. Most of them have been established in \cite{HK-tams} or \cite{HKR15}.
\subsection{On a half line}
As simplest model, we start with the operator, acting on $L^2(\R_+)$, defined by
\begin{equation}\label{defH00}
\mathcal H_{0}=-\partial^2_{\tau}
\end{equation}
with domain
\begin{equation}
\Dom(\mathcal H_{0})=\{u\in H^2(\R_+)~:\,u'(0)=- u(0)\}\,.
\end{equation}
Note that this operator is associated with the quadratic form
\[
V_0 \ni u\mapsto \int_0^{+\infty} |u'(\tau)|^2\,d\tau\,  - |u (0)|^2\,,
\]
with $V_0 = H^1(0,+\infty)\,$.

The spectrum of this operator is
$\{-1\}\cup[0,\infty)$.
The eigenspace of the eigenvalue $-1$ is generated by the $L^2$-normalized function
\begin{equation}\label{eq:u0}
u_0(\tau)=\sqrt{2}\,\exp\left(-\tau\right)\,.
\end{equation}
We will also consider this operator in a bounded interval $(0,T)$ with $T$ sufficiently large and Dirichlet condition at $\tau=T$.

\subsection{On an interval}
Let us consider $T\geq 1$ and the self-adjoint operator acting on $ L^2(0,T)$ and defined by
\begin{equation}\label{eq:H0}
\mathcal H^{\{T\}}_{0}=-\partial^2_{\tau}\,,
\end{equation}
with domain,
\begin{equation}\label{eq:DomH0}
\Dom(\mathcal H^{\{T\}}_{0})=\{u\in H^2(0,T)~:~u'(0)=-u(0)\quad{\rm and}\quad u(T)=0\}\,.
\end{equation}
The spectrum of the operator $\mathcal H^{\{T\}}_{0}$ is purely discrete
and consists of a strictly increasing sequence of eigenvalues
denoted by $\left(\lambda_n\left(\mathcal H^{\{T\}}_{0}\right)\right)_{n\geq 1}$. This operator is associated with the quadratic form
\[
V^{\{T\}}_0\ni u\mapsto \int_0^{T} |u'(\tau)|^2\,d\tau\,  - |u (0)|^2\,,
\]
with $V^{\{T\}}_0 =\{v\in H^1(0,T)\,|\, v(T)=0\}$.\\
The next lemma gives the  localization of  the two first eigenvalues
$\lambda_1\left(\mathcal H^{\{T\}}_{0}\right)$ and $\lambda_2\left(\mathcal H^{\{T\}}_{0}\right)$ for large values of $T$.

\begin{lemma}\label{lem:1DL}
As $T \to +\infty$, there holds
\begin{equation}\label{eq:lwh}
\lambda_1(\mathcal H^{\{T\}}_{0})= - 1 + 4 \big(1+o(1)\big) \exp\big( - 2 T\big)\quad{\rm
and}\quad \lambda_2(\mathcal H^{\{T\}}_{0})\geq 0\,.
\end{equation}
Let us now discuss the estimates of the next eigenvalues.
\end{lemma}
\begin{lemma}\label{lem:1DL*}
For all $T>1$ and $n\geq 2$, 
\[\left(\frac{(2n-3)\pi}{2T}\right)^2<\lambda_n(\mathcal H_{0}^{\{T\}})<\left(\frac{(n-1)\pi}{T}\right)^2\,.\]
\end{lemma}
\begin{proof}
Let $w\geq0$ and $\lambda={\Bl -}w^2$ be a non-negative  eigenvalue of the
operator $\mathcal H_{0}^{\{T\}}$ with an eigenfunction $u$. We have,
\begin{equation} \label{ab1}
-u''=\lambda u\quad{\rm in~}(0,T)\,,\quad
u'(0)=-u(0)\,,\quad u(T)=0\,.
\end{equation}
 If $w=0$ and $T>1$, then $u=0$ is the unique
 solution of \eqref{ab1}.  Thus, $w>0$ and
\begin{equation} \label{ab2}
u(\tau)=A\cos(w\tau)+B\sin(w\tau)\,,
\end{equation}
for some constants $A\in\R$ and $B\in\R$ that depend on $T$. The boundary conditions satisfied by $u$ yield that $A=-Bw$, $\cos(wT)\not=0$ and
\[\tan(wT)=w\,.\]
Thus $w$ is a fixed point of the $\pi/T$-periodic function $x\mapsto\tan(xT)$. Obviously, there exist infinitely many solutions, at least one solution in every interval $(-\frac{\pi}{2T},\frac{\pi}{2T})+\frac{k\pi}{T}$, $k=0,\pm1,\cdots$. Since we are interested in the  positive solutions, we specialize first into the interval $(-\frac{\pi}{2T},\frac{\pi}{2T})$. Define the function $g(x)=\tan(xT)-x$. Clearly, $x=0$ is a zero of this function in the interval $(-\frac{\pi}{2T},\frac{\pi}{2T})$. It is the unique zero of $g$ in this interval since $g'(x)=T(1+\tan^2(xT))-1>0$ for $T>1$. Thus, the smallest $w>0$ that satisfies $g(w)=0$ does live in the interval $(\frac{\pi}{2T},\frac{\pi}{T})$, which is $\sqrt{\lambda_2(\mathcal H_{0}^{\{T\}})}$. The next positive zero of $g$, $\sqrt{\lambda_3(\mathcal H_{0}^{\{T\}})}$,  lives in the interval $(\frac{\pi}{2T},\frac{\pi}{T})+\frac{\pi}T$, etc.
\end{proof}

\subsection{In a weighted space}
Let $B\in\R$, $T >0$ such that $|B|T < \frac 13$. Consider the self-adjoint
operator, acting on $L^2\big((0,T);(1-B\tau)\dx\tau\big)$ and defined by
\begin{equation}\label{eq:H0b}
\mathcal H^{\{T\}}_{B}=-(1-B\tau)^{-1}\partial_{\tau}(1-B\tau)\partial_{\tau}=-\partial^2_{\tau}+B(1-B\tau)^{-1}\partial_{\tau}\,,
\end{equation}
with domain
\begin{equation}\label{eq:domH0b}
\Dom(\mathcal H^{\{T\}}_{B})=\{u\in H^2(0,T)~:~u'(0)=-u(0)\quad{\rm and}\quad u(T)=0\}\,.
\end{equation}
The operator $\mathcal H^{\{T\}}_{B}$ is the Friedrichs extension in $L^2\big((0,T);(1-B\tau)\dx\tau\big)$ associated
with the quadratic form defined for $u\in V^{\{T\}}_h$, by
\[q^{\{T\}}_{B}(u)=\int_0^{T}|u'(\tau)|^2(1-B\tau)\dx\tau-|u(0)|^2\,.\]
The operator $\mathcal H^{\{T\}}_{B}$ is with compact resolvent. The strictly increasing
sequence of the eigenvalues of $\mathcal H^{\{T\}}_{B}$ is denoted by
$(\lambda_n(\mathcal H^{\{T\}}_{B})_{n\in \mathbb N^*}$.
It is easy to compare the spectra of $\mathcal H^{\{T\}}_{B}$ and  $\mathcal H^{\{T\}}_{0}$ as $B$ goes to $0$.
\begin{lemma}\label{lem:H0b}
There exist $T_0, C >0$ such that for all $T \geq T_0$, $B\in \left(-1/(3T), 1/(3T)\right)$ and $n\in\mathbb N^*$, there holds,
\[\left|\lambda_n(\mathcal H^{\{T\}}_{B})-\lambda_n(\mathcal H^{\{T\}}_{0})\right|\leq C|B|T\Big(\,\big|\lambda_n(\mathcal H^{\{T\}}_{0})\big|+1\Big)\,.\]
\end{lemma}
Then we notice that, for all $T>0$, the family $\left(\mathcal{H}^{\{T\}}_{B}\right)_{B}$ is analytic for $B$ small enough. More precisely, we have

\begin{lemma}\label{lem.analyticB}
There exist $T_0 >0$ such that for all $T \geq T_0$, the two functions $(-1/(3T), 1/(3T)) \ni B\mapsto \lambda_{1}\left(\mathcal{H}^{\{T\}}_{B}\right)$ and  $(-1/(3T), 1/(3T)) \mapsto u_{B}^{\{T\}}$ are analytic. Here $u_{B}^{\{T\}}$   is the corresponding positive and normalized eigenfunction $ \lambda_{1}\left(\mathcal{H}^{\{T\}}_{B}\right)$.
\end{lemma}

The next proposition states a two-term asymptotic expansion of the eigenvalue $\lambda_1(\mathcal H^{\{T\}}_{B})$.
\begin{proposition}\label{lem:H0b;l} There exist $T_0 >0$ and $C>0$  such that for all $T \geq T_0$, for all $B\in \left(-1/(3T), 1/(3T)\right)$ there holds,
\[\Big|\lambda_1(\mathcal H^{\{T\}}_{B})-(-1-B)\Big|\leq CB^2\,.\]
\end{proposition}
We have also a decay estimate of $u_{B}^{\{T\}}$ that is a classical consequence of Proposition \ref{lem:H0b;l}, the fact that the Dirichlet problem on $(0,T)$ is positive  and of  Agmon estimates.

\begin{proposition}\label{prop.AgmonuBT}There exist $T_0 >0$, $\alpha >0$ and $C>0$  such that for all $T \geq T_0$, for all $B\in \left(-1/(3T), 1/(3T)\right)$  there holds,
\[\|e^{\alpha\tau}u_{B}^{\{T\}}\|_{H^1\big((0,T);(1-B\tau)\dx\tau\big)}\leq C\,.\]
\end{proposition}

\begin{lemma}\label{lem.h12}
There exist $C>0$  and $T_0 >0$ such that for all $T\geq T_0$ and all $B\in(-1/(3T), 1/(3T))$,
\begin{align}
\left|\partial_{B}\lambda_{1}\left(\mathcal{H}_{B}^{\{T\}}\right)\right|&\leq C\label{FHb}\,,\\
\|\partial_{B} \tilde u_{B}^{\{T\}}\|_{L^2((0,T),\dx\tau)}&\leq C\label{FHc}\,.
\end{align}
where $\tilde u_{B}^{\{T\}}=(1-B\tau)^{\frac{1}{2}}u_{B}^{\{T\}}$.
\end{lemma}

\subsection*{Acknowledgments}
The authors would like to thank K. Pravda-Starov for stimulating discussions. This work was partially supported by the Henri Lebesgue Center (programme  \enquote{Investissements d'avenir}  -- n$^{\rm o}$ ANR-11-LABX-0020-01).

\end{document}